\documentclass[10pt]{amsart}
\usepackage{enumerate,amsmath,amssymb,latexsym,
amsfonts, amsthm, amscd, MnSymbol}

\theoremstyle{plain}

\newtheorem{theorem}{Theorem}

\numberwithin{equation}{section}

\newcommand{\sm}{\rm sm}
\newcommand{\cm}{\rm cm}
\newcommand{\sll}{\rm sl}

\newcommand{\ii}{\rm i}

\addtocounter{section}{-1}

\begin{document}

\title {The Elliptic Functions in a First-Order System}

\date{}

\author[P.L. Robinson]{P.L. Robinson}

\address{Department of Mathematics \\ University of Florida \\ Gainesville FL 32611  USA }

\email[]{paulr@ufl.edu}

\subjclass{} \keywords{}

\begin{abstract}

We investigate the first-order system `$s\,' = c^3, \, c\,' = - s^3; \, s(0) = 0, \, c(0) = 1$'. Its solutions have the property that $s \, c$, $s^2$ and $c^2$ extend to simply-poled elliptic functions, which we explicitly identify in terms of  the lemniscatic Weierstrass $\wp$ function. 

\end{abstract}

\maketitle

\medbreak

\section*{Introduction} 

The first-order initial value problem 
$$s\,' = c, \; c\,' = -s; \; \; s(0) = 0, \; c(0) = 1$$
needs no introduction: its solutions are entire, being the familiar trigonometric  functions $s = \sin$ and $c = \cos$. 

\medbreak 

The first-order initial value problem 
$$s\,' = c^2, \; c\,' = -s^2; \; \; s(0) = 0, \; c(0) = 1$$
is less well-known: its solutions are the elliptic functions $s = \sm$ and $c = \cm$ of Dixon [2]. Long after their discovery, these Dixonian elliptic functions were found to have interesting connexions to combinatorics [1] and geometry [5]. 

\medbreak 

Here we propose to study the `next' first-order initial value problem 
$$s\,' = c^3, \; c\,' = -s^3; \; \; s(0) = 0, \; c(0) = 1.$$
The standard Picard theorem ensures that this system has a unique holomorphic solution pair in a suitably small disc about $0$. It is easy to see that solutions to this system cannot be meromorphic throughout the plane; in particular, the solutions to the system are not elliptic. However, elliptic functions are still present: in fact, the quadratic combinations $s c$, $s^2$ and $c^2$ are elliptic; indeed, we shall here identify them in terms of the lemniscatic Weierstrass $\wp$ function and the associated lemniscatic sine function $\sll$ of Gauss. For these lemniscatic functions, see Chapter XXII of [6]. 

\medbreak 

\section*{A First-Order System} 

\medbreak 

This paper concerns the initial value problem 
$$s\,' = c^3, \; c\,' = -s^3; \; \; s(0) = 0, \; c(0) = 1$$
which will henceforth be referred to as {\bf IVP}. In this section, the discussion will be largely local: we primarily consider properties of the unique solutions to {\bf IVP} in an appropriate disc about the origin. The global questions involved in extending these solutions will be settled in the following section: as we shall see, the solutions themselves do not extend meromorphically, but their product and their squares actually extend as elliptic functions. 

\medbreak 

First of all, we observe that solutions to {\bf IVP} satisfy an identity akin to the `Pythagorean' trigonometric identity. 

\medbreak 

\begin{theorem} \label{Pyth}
Solutions $s$ and $c$ to {\bf IVP} in any connected open set containing $0$ satisfy 
$$s^4 + c^4 = 1.$$ 
\end{theorem} 

\begin{proof} 
Differentiation yields 
$$(s^4 + c^4)\,' = 4 s^3 (s\,') + 4 c^3 (c\,') = 4 s^3 (c^3) + 4 c^3 (-s^3) = 0$$
and evaluation at $0$ places the constant value at $1$. 
\end{proof} 

\medbreak 

More generally, setting aside the initial conditions, if functions $s$ and $c$ satisfy the differential equations $s\,' = c^3$ and $c\,' = - s^3$ then $s^4 + c^4$ is constant on any connected open set. 

\medbreak 

The Picard existence-uniqueness theorem ensures that {\bf IVP} has a unique pair of holomorphic solutions $(s, c)$ in a sufficiently small disc about the origin. For the sake of completeness, we shall offer an appropriate disc; however, the size of such a disc is immaterial for our present purposes. As references for the Picard theorem we cite [3] Section 2.3 and [4] Section 3.3 (modified for complex use as mentioned in Section 12.1).

\medbreak 

\begin{theorem} \label{Picard}
The system {\bf IVP} has a unique holomorphic solution pair $(s, c)$ in the disc $\mathbb{D}_r$ of radius $r = 2^2/3^3$ about $0$. 
\end{theorem} 

\begin{proof} 
 Fix $b > 0$: for $|s| \leqslant b$ and $|c - 1| \leqslant b$ we have $|s^3| \leqslant b^3$ and $|c^3| \leqslant (b + 1)^3$; as the system {\bf IVP} is autonomous, the Picard theorem ensures that this system has a unique holomorphic solution in the disc of radius $b/(b + 1)^3$ about the origin. The radius of this disc is maximized by taking $b = 1/2$: thus {\bf IVP} has a unique holomorphic solution in the disc of radius $r = 2^2/3^3$ about $0$.
\end{proof} 

\medbreak 

Incidentally, we may use the `Pythagorean' identity of Theorem \ref{Pyth} to extract from {\bf IVP} a corresponding initial value problem for $s$ alone: namely, 
$$s\,' = (1 - s^4)^{3/4}, \; \; s\,'(0) = 1.$$
Here, the initial condition $s\,'(0) = 1$ both forces $s(0) = 0$ and specifies the (principal) branch of the power; the initial condition $s(0) = 0$ alone would be inadequate. A similar application of the Picard theorem to this initial value problem furnishes a unique holomorphic solution in the larger disc of radius $(2^2/3^3)^{1/4}$ about $0$; this solution $s$ takes values in the open unit disc, so a companion holomorphic $c$ is provided by $(1 - s^4)^{1/4}$ with principal-valued power. 

\medbreak 

The solutions $s$ and $c$ to {\bf IVP} in the disc $\mathbb{D}_r$ of radius $r = 2^2/3^3$ about $0$ satisfy certain symmetry properties. Under multiplication by the imaginary unit $\ii$ they behave as follows. 

\medbreak 

\begin{theorem} \label{i}
If $|z| < r$ then $s(\ii z) = \ii s(z)$ and $c(\ii z) = c(z)$. 
\end{theorem} 

\begin{proof} 
Define functions $f$ and $g$ by $f(z) = - \ii s(\ii z)$ and $g(z) = c(\ii z)$. By differentiation, $f\,' = g^3$ and $g\,' = - f^3$; by evaluation, $f(0) = 0$ and $g(0) = 1$. Now appeal to the uniqueness clause in Theorem \ref{Picard} to conclude that $f = s$ and $g = c$. 
\end{proof} 

\medbreak 

A similar argument shows that $s$ and $c$ are `real' in the sense that if $|z| < r$ then $s(\overline{z}) = \overline{s(z)}$ and $c(\overline{z}) = \overline{c(z)}$. As an immediate corollary to Theorem \ref{i}, $s$ is an odd function and $c$ is an even function. 

\medbreak 

We now wish to explore the possibility of extending the solutions $s$ and $c$ of {\bf IVP} beyond the disc $\mathbb{D}_r$ provided by Theorem \ref{Picard} or beyond the larger disc provided by the subsequent remark. Of course, if these functions are extended to a suitably symmetric connected open set then the symmetries just discussed will continue to hold by the principle of analytic continuation. 

\medbreak 

The fundamental question to be answered is whether the extension of $s$ and $c$ runs up against singularities (be they poles or otherwise) or the necessity of forming branches. We can answer this question with regard to poles decisively and at once. 

\medbreak 

\begin{theorem} \label{nopoles}
Isolated singularities of functions $s$ and $c$ that satisfy $s\,' = c^3$ and $c\,' = - s^3$ cannot be poles. 
\end{theorem} 

\begin{proof} 
The functions $s$ and $c$ are plainly copolar, in the sense that if either has a pole at some point then so does the other. Consider a pole: of order $m$ for $s$ and of order $n$ for $c$. From $s\, = c^3$ follows $m + 1 = 3 n$; from $c\,' = - s^3$ follows $n + 1 = 3 m$. These two equations for $m$ and $n$ have $1/2$ as their unique solution. 
\end{proof}

\medbreak 

This proof suggests the possibility that one or more of the products $s c$, $s^2$ and $c^2$ might admit extension as a meromorphic function with simple poles. In the next Section we realize this possibility and improve upon it, thereby simultaneously addressing both singularities and branches. 

\medbreak 

\section*{The Elliptic Functions}

\medbreak 

We begin with the product of the functions $s$ and $c$ appearing in Theorem \ref{Picard}: thus, let $p = s \, c$. By differentiation, 
$$p\,' = s\,' \, c + s \, c\,' = c^4 - s^4$$
so that 
$$(p\, ')^2 = (c^4 + s^4)^2 - 4 s^4 c^4 = 1 - 4 \, p^4$$
(by Theorem \ref{Pyth}) and 
$$p\, '' = 4 c^3 c\,' - 4 s^3 s\,' = - 8 s^3 c^3  = - 8 \, p^3.$$
Rescale: define $P$ by the rule that if $|z| < r \sqrt2$ then 
$$P (z) = \sqrt2 \, p(\tfrac{1}{\sqrt2} \, z).$$ 
We conclude that $P$ satisfies the following second-order initial value problem: 
$$P\,'' = - 2 P^3 ; \; \; P(0) = 0, \; P\,'(0) = 1.$$ 
This enables us to relate the product $s \, c$ to the lemniscatic sine function $\sll$ of Gauss and so in turn to the Glaisher quotient ${\rm s d} = {\rm sn}/{\rm dn}$ of the Jacobian functions with self-complementary modulus $1/\sqrt2$. 

\medbreak 

\begin{theorem} \label{sc}
If $|z| < r$ then $s(z)  c(z) = \tfrac{1}{\sqrt2} \, \sll (\sqrt2 z) = \tfrac{1}{2} \, {\rm sd} ( 2 z).$
\end{theorem} 

\begin{proof} 
The second-order initial value problem displayed immediately prior to the Theorem characterizes the lemniscatic sine function: $P$ coincides with $\sll$ on a neighbourhood of $0$ by the Picard existence-uniqueness theorem and thence throughout the disc $\mathbb{D}_{r \sqrt2}$ of radius $r \sqrt2$ about $0$ by the principle of analytic continuation; all that remains is to undo the rescaling and recall the expression for $\sll$ in terms of ${\rm sd}$.  
\end{proof} 

\medbreak 

That is, the product $s \, c$ agrees with the elliptic function $z \mapsto \tfrac{1}{2} {\rm sd}(2 \, z)$ on the disc $\mathbb{D}_r$; in other words, $s \, c$ extends to this elliptic function by analytic continuation. In short, we may simply say that $s \, c$ is this elliptic function, taking similar liberties in Theorem \ref{c2} and Theorem \ref{s2}. 

\medbreak

 We now turn to the squares $S = s^2$ and $C = c^2$ of $s$ and $c$. We take $C$ first, as this case is a little more straightforward. By differentiation, 
$$C\,' = 2 c \, c\,' = - 2 c \, s^3$$
whence Theorem \ref{Pyth} yields 
$$(C\,')^4 = 16 c^4 (s^4)^3 = 16 c^4 (1 - c^4)^3.$$ 
Thus, $C$ satisfies the first-order differential equation 
$$(C\,')^4 = 16 \, C^2 \, (1 - C^2)^3$$ 
of fourth degree, along with the initial condition $C(0) = 1$. This is a differential equation of Briot-Bouquet type: see page 423 of [3] and page 314 of [4]; it may be solved as follows. 

\medbreak 

First, the substitution $E = C^{-1}$ has the effect of removing the quadratic factor from the right side: explicitly, $E\,' = - C^{-2} C\,'$ so that 
$$(E\,')^4 = C^{-8}\, (C\,')^4 = C^{-8} \, 16\, C^2 \, (1 - C^2)^3 = 16 \, C^{-6} \, (1 - C^2)^3 = 16 \, (C^{-2} - 1)^3$$
and therefore 
$$(E\,')^4 = 16 \, (E^2 - 1)^3 = 16 \, (E - 1)^3 \, (E + 1)^3.$$
Next, the substitution $F = (E - 1)^{-1}$ yields 
$$(F\,')^4 = 128 \, F^2 (F + \tfrac{1}{2})^3$$ 
and the substitution $G^2 = F + \tfrac{1}{2}$ leads to   
$$(G\,')^4 = 8 \, (G^2 - \tfrac{1}{2})^2 \, G^2$$ 
whence 
$$(G\,')^2 = 2 \sqrt2 G (G^2 - \tfrac{1}{2})$$
after a choice of square-root. Now make a final substitution $H = \tfrac{1}{\sqrt2} \, G$: reversal of the various substitutions reveals that 
$$C = \frac{4 \, H^2 - 1}{4 \, H^2 + 1} \, .$$ 
\medbreak 
\noindent 
On the one hand, $H$ satisfies the first-order differential equation 
$$(H\,')^2 = 4 \, H^3 - H; $$
on the other hand, $H$ has a pole at the origin since $C(0) = 1$. These conditions force $H$ to be the {\it lemniscatic} Weierstrass function $\wp$ with invariants $g_2 = 1$ and $g_3 = 0$. 

\medbreak 

\medbreak 

\begin{theorem} \label{c2}
The square $C = c^2$ is given by 
$$C = \frac{\wp^2 - 1/4}{\wp^2 + 1/4}$$ 
where $\wp$ is the Weierstrass function with invariants $g_2 = 1$ and $g_3 = 0$. 
\end{theorem} 

\begin{proof} 
Essentially a reversal of the arguments that led to the theorem. In the result, the Weierstrass function may be replaced by its negative, since only its square appears. In the derivation, this ambiguity arises from the substitution $G^2 = F + \tfrac{1}{2}$ and the attendant choice of square-root. 
\end{proof} 

\medbreak 

The analysis for $S$ runs parallel to that for $C$ until the point at which the initial condition intervenes. From $S\,' = 2 s \, c^3$ it follows that 
$$(S\,')^4 = 16 \, S^2 \, (1 - S^2)^3$$
and a repetition of the foregoing argument shows that 
$$S = \frac{4 \, H^2 - 1}{4 \, H^2 + 1}$$
where 
$$(H\,')^2 = 4 \, H^3 - H.$$
This requires that $H$ be a translate of the Weierstrass function $\wp$ that appears in Theorem \ref{c2}: thus, for some $k$ and for all $z$, 
$$S(z) =  \frac{ \wp(z - k)^2 - 1/4}{ \wp(z - k)^2 + 1/4}.$$ 
The initial condition $S(0) = 0$ forces $\wp(k)^2 = 1/4$ so that $\wp(k)$ is a mid-point value $\pm 1/2$ of the lemniscatic $\wp$: modulo periods, $k$ is either the real half-period 
$$\omega = 2 \int_0^1 \frac{{\rm d} \tau \; \; \; }{(1 + \tau^4)^{1/2}} = 1.85407467730 ... $$ 
for which $\wp(\omega) = 1/2$ or the purely imaginary half-period $\ii \, \omega$ for which $\wp(\ii \, \omega) =  - 1/2$; we claim that the former half-period is appropriate, rather than the latter. Direct calculation (for instance, using the $\wp$ addition formula) reveals that 
$$\wp(z - \omega) = \frac{1}{2} \, \frac{\wp(z) + 1/2}{\wp(z) - 1/2}$$
whence 
$$\frac{ \wp(z - \omega)^2 - 1/4}{ \wp(z - \omega)^2 + 1/4} = \frac{\wp(z)}{\wp(z)^2 + 1/4}$$
and therefore 
$$\frac{ \wp(z - \ii \, \omega)^2 - 1/4}{ \wp(z - \ii \, \omega)^2 + 1/4} = - \, \frac{\wp(z)}{\wp(z)^2 + 1/4}$$
\medbreak
\noindent
because the lemniscatic $\wp$ satisfies $\wp(\ii \, z) = - \wp(z)$. Finally, we see that $k = \omega$ must be chosen: for if $t$ is real then (as noted after Theorem \ref{i}) $s(t)$ is real so that $S(t) = s(t)^2 \geqslant 0$. 

\medbreak 

\begin{theorem} \label{s2}
The square $S = s^2$ is given by 
$$S = \frac{\wp}{\wp^2 + 1/4}$$
where $\wp$ is the Weierstrass function with invariants $g_2 = 1$ and $g_3 = 0$.
\end{theorem} 

\begin{proof} 
Above. 
\end{proof} 

\medbreak 

It is perhaps barely worth noting that the expressions for $C$ and $S$ in Theorem \ref{c2} and Theorem \ref{s2} satisfy $S^2 + C^2 = 1$ as they should. 

\medbreak 

It is certainly worth noting that $s^2$ and $c^2$ satisfy the second-order system 
$$S\,'' = 2 \, C^3 - 6 \, S^2 \, C \; \; {\rm and} \; \; C\,'' = 2 S^3 - 6 \, C^2 \, S$$ 
with 
$$S(0) = 0, \, S\,'(0) = 0 \; \; {\rm and} \; \; C(0) = 1, \, C\,'(0) = 0$$
as do the rational functions of the lemniscatic $\wp$ displayed in Theorem \ref{c2} and Theorem \ref{s2}. The right members of these differential equations being polynomial in $S$ and $C$, the uniqueness clause in the Picard theorem applies to justify the conclusions of Theorem \ref{c2} and Theorem \ref{s2}. We may instead proceed to fourth order, each of $S$ and $C$ being a solution to the differential equation 
$$F\,'''' = -  12 \, F \, (32  F^4 - 40 F^2 + 9).$$

\medbreak 

It is also worth noting that the results displayed in  Theorem \ref{sc}, Theorem \ref{c2} and Theorem \ref{s2} are consistent. That this is so follows from special properties of the lemniscatic $\wp$: for this Weierstrass function, the Glaisher quotient ${\rm sd}$ satisfies 
$${\rm sd}^2 = \frac{1}{\wp}$$
and the duplication formula for $\wp$ gives 
$$\wp(2 z) = \frac{(\wp(z)^2 + 1/4)^2}{\wp\,'(z)^2}$$
whence  
$$\Big(\frac{1}{2} \, {\rm sd}(2 z)\Big)^2 = \frac{1}{4} \, \Big( \frac{\wp\,'(z)}{\wp(z)^2 + 1/4} \Big)^2 =  \frac{ \wp(z)^2 - 1/4}{ \wp(z)^2 + 1/4} \, \frac{\wp(z)}{\wp(z)^2 + 1/4}. \, $$

\medbreak 

In fact, we can go further and extract the square-root: an examination of behaviour as $z \to 0$ reveals that 
$${\rm sd} (2 z) = - \, \frac{\wp\,'(z)}{\wp(z)^2 + 1/4}\, .$$
As a consequence, we may now return to Theorem \ref{sc} and express the product $s \, c$ directly in terms of $\wp$: thus, 
$$s \, c = - \, \frac{1}{2} \, \frac{\wp\,'}{\wp^2 + 1/4} \,.$$ 

\medbreak 

\medbreak 

Finally, we realize the possibility that opened up at the close of the previous Section. It follows from Theorem \ref{sc} (or the reformulation just given), Theorem \ref{c2} and Theorem \ref{s2} that each of $s \, c$, $c^2$ and $s^2$ extends to an elliptic function having (simple) poles exactly at the points congruent to $\tfrac{1}{2} (\pm 1 \pm \ii) \, \omega$ modulo the periods $2 \, \omega$ and $2 \, \ii  \, \omega$ of the lemniscatic Weierstrass function $\wp$. Notice that each of these three elliptic functions is holomorphic in the open disc of radius $\omega / \sqrt2$ about $0$. Within this disc, the function $\wp / (\wp^2 + 1/4)$ has a double zero at $0$ (originating as a removable singularity) but is otherwise nonzero; in this disc it therefore has two holomorphic square-roots, which extend the functions $s$ and $-s$. The simple poles of the function $\wp / (\wp^2 + 1/4)$ at the points $\tfrac{1}{2} (\pm 1 \pm \ii) \, \omega$ on the boundary of this disc sprout branches when the attempt is made to extend the holomorphic function $s$ beyond them. The situation as regards the function $c$ is slightly simpler, the function $(\wp^2 - 1/4) / (\wp^2 + 1/4)$ being zero-free in the same disc. 

\medbreak

\bigbreak

\begin{center} 
{\small R}{\footnotesize EFERENCES}
\end{center} 
\medbreak 

[1] E. Conrad and P. Flajolet, {\it The Fermat cubic, elliptic functions, continued fractions, and a combinatorial excursion}, S\'eminaire Lotharingien de Combinatoire 54 (2006) Article B54g.

\medbreak 

[2] A.C. Dixon, {\it On the doubly periodic functions arising out of the curve $x^3 + y^3 - 3 \alpha x y = 1$}, The Quarterly Journal of Pure and Applied Mathematics, 24 (1890) 167-233. 

\medbreak 

[3] E. Hille, {\it Ordinary Differential Equations in the Complex Domain}, Wiley-Interscience (1976); Dover Publications (1997).

\medbreak 

[4] E.L. Ince. {\it Ordinary Differential Equations}, Longman, Green and Company (1926); Dover Publications (1956). 

\medbreak 

[5] J.C. Langer and D.A. Singer, {\it The Trefoil}, Milan Journal of Mathematics, 82 (2014) 161-182. 

\medbreak 

[6] E. T. Whittaker and G. N. Watson, {\it A Course of Modern Analysis}, Second Edition, Cambridge University Press (1915).

\medbreak

\end{document}